\documentclass[a4paper,12pt,reqno]{amsart}
\usepackage{amsfonts,amssymb,hyperref,amsthm,enumerate,
	color,stmaryrd,pgf,tikz,comment,multirow}
\usepackage{amsmath}
\usepackage{tikz-cd}
\usepackage[left=2.6cm,right=2.6cm,top=3cm,bottom=3cm,bindingoffset=0cm]{geometry}
\usepackage{booktabs,caption}
\usepackage[flushleft]{threeparttable}
\numberwithin{equation}{section}

\newtheorem{theorem}{Theorem}[section]
\newtheorem{lemma}[theorem]{Lemma}
\newtheorem{proposition}[theorem]{Proposition}

\theoremstyle{definition}
\newtheorem{remark}[theorem]{Remark}

\newtheorem{corollary}[theorem]{Corollary}
\newtheorem{problem}[theorem]{Problem}
\newtheorem{question}[theorem]{Question}

\newtheorem*{case}{Case}

\DeclareMathOperator{\Ker}{Ker}
\DeclareMathOperator{\Core}{Core}

\DeclareMathOperator{\Aut}{Aut}

\DeclareMathOperator{\Mon}{Mon}

\DeclareMathOperator{\av}{av}
\DeclareMathOperator{\ord}{ord}

\title[Skew morphisms of skew-type four on cyclic $2$-groups]%
{Classification and enumeration of skew morphisms of skew-type four on cyclic $2$-groups}
\author[K. Hu]{Kan Hu$^{*}$}
\address{K. Hu,
\newline\indent
Department of Mathematics, Zhejiang Ocean University, Zhoushan, Zhejiang 316022, P.R. China
}
\email{hukan@zjou.edu.cn}
\thanks{This work was supported by National Natural
Science Foundation of China (12471332).
\newline\indent
$^{\ast}$ Corresponding author e-mail: hukan@zjou.edu.cn
}

\begin{document}
\maketitle

\begin{abstract}
A skew morphism on a finite group $A$ is a permutation $\varphi$ on $A$ that fixes the identity element
of $A$ and for which there exists an integer-valued function $\pi:A\to\mathbb{Z}_{|\varphi|}$ such that
$\varphi(xy)=\varphi(x)\varphi^{\pi(x)}(y)$ for all $x,y\in A$. The kernel of $\varphi$
is the subgroup $\Ker\varphi=\{x\in A\mid \pi(x)=1\}$, and the index $[A:\Ker\varphi]$ is called
the skew-type of $\varphi$. In this paper we construct, classify and enumerate the
skew morphisms of skew-type four on cyclic $2$-groups. Our main results give explicit
formulas for all such skew morphisms and closed-form expressions for their numbers.
\end{abstract}

\section{Introduction}

A \textit{map} $M$ is an embedding $i:\Gamma\hookrightarrow\mathbb{S}$ of a connected
graph $\Gamma$ into a closed surface $\mathbb{S}$
such that each component of $\mathbb{S}\backslash i(\Gamma)$ is homeomorphic to the unit open disc.
A map is orientable if its supporting surface is orientable; otherwise it is non-orientable.
Throughout the paper, all maps considered are orientable.
An \textit{automorphism} of an orientable map $M$ is an automorphism of the embedded graph $\Gamma$
which extends to an orientation-preserving self-homeomorphism of the supporting surface $\mathbb{S}$.
The set $\Aut(M)$ of all automorphisms of $M$ forms the automorphism group of $M$ under composition.

To study the symmetry of a map $M$, it is common to introduce a combinatorial description of $M$:
$M$ can be identified with a triple $(\Omega,\rho,\lambda)$, where $\Omega$ is the
set of arcs of $M$, $\rho$ is the product of the local orientations around each vertex $v$, induced by the
global orientation of the supporting surface $\mathbb{S}$,
which cyclically permutes the arcs emanating from $v$, and $\lambda$ is the involution
transposing the arcs incident with a common edge of the embedded graph $\Gamma$.
In this way, an automorphism of $M$ can be identified with a permutation on $\Omega$ which commutes with
both $\rho$ and $\lambda$, so that the automorphism group $\Aut(M)$, being the centralizer
of the transitive permutation group $\Mon(M)=\langle\rho,\lambda\rangle$, acts semi-regularly on $\Omega$.
If this action is also transitive and hence regular, the map itself is called \textit{regular}.

A map $M$ is called a Cayley map if $\Aut(M)$ contains a subgroup $A$ acting regularly on its vertices.
For a Cayley map $M=(\Omega,\rho,\lambda)$ with a simple underlying graph $\Gamma$,
we may identify the set of vertices of $M$ with the subgroup $A$, so that $\Gamma=(V,E)$,
where $V=A$ and $E=\big\{\{x,y\}\mid x,y\in A,\ xy^{-1}\in S\big\}$,
and the map $M$ can be described as a triple $(A,S,p)$, where $S$ is the set of vertices
adjacent to the identity $1_A$, and $p$ is the cyclic permutation on $S$ corresponding to the local
orientation around the vertex $1_A$, with $\rho$ and $\lambda$ specified by
\[
\rho(x,xs)=(x,xp(s))\quad\text{and}\quad\lambda(x,xs)=(xs,x),
\]
for all $x,y\in A$ and $s\in S$. In particular, as shown by Jajcay and \v{S}ir\'{a}\v{n} in their seminal
paper~\cite[Theorem 1]{JS2002}, $M$ is regular if and only if $p$ extends to a skew morphism
$\varphi$ on $A$. Thus, for a given finite group $A$, the study of regular Cayley maps on $A$
is essentially the study of \textit{Cayley skew morphisms} on $A$, namely,
skew morphisms which contain an inverse-closed orbit generating $A$.

The second motivation to study skew morphisms is related to group factorizations.
A \textit{cyclic complementary extension} of a finite group $A$ is a finite group
$G$ which contains $A$ as a subgroup and admits a factorization into an exact product
$G=AC$ of $A$ and a cyclic complementary subgroup $C$.
If we choose and fix a generator $c$ of the cyclic subgroup $C$, then, for every $x\in A$, there exists
a unique element $\varphi(x)\in A$ and a unique integer $\Pi(x)\in\mathbb{Z}_{|c|}$ such that
$cx=\varphi(x)c^{\Pi(x)}$. It is easily seen that $\varphi:A\to A$ is a skew morphism
on $A$ and $\Pi:A\to\mathbb{Z}_{|c|}$ reduces to the power function $\pi: A\to\mathbb{Z}_{|\varphi|}$
via the congruence $\pi(x)\equiv\Pi(x)\pmod{|\varphi|}$.
The function $\Pi$, known as an \textit{extended power function} of $\varphi$, together with
the skew morphism $\varphi$, can be used to recover the group $G$ in a canonical
way~\cite[Theorem 1]{HJ2026}. Therefore, for a given finite group $A$, the study of
cyclic complementary extensions of $A$ is essentially the study of skew morphisms on $A$ and the
associated extended power functions.

Skew morphisms have been extensively investigated in connection with
regular Cayley maps and cyclic complementary extensions of finite groups.
To the best of our knowledge,
regular Cayley maps have been classified completely for the cyclic groups~\cite{CT2014},
the dihedral groups~\cite{KK2016,KK2017,KK2021,KKF2006,KMM2013,WF2005,Zhang2015a,Zhang2015b},
and the elementary abelian groups~\cite{DYL2023},
and partially for other abelian groups~\cite{CJT2007,CJT2007b},
certain nonabelian metacyclic groups such
as semi-dihedral groups~\cite{Oh2009,HQ2026} and dicyclic groups~\cite{KO2008}, and others.
On the other hand, cyclic complementary extensions have been investigated for the monolithic groups
(including the finite non-abelian simple groups)~\cite{BCV2022},
the characteristically simple groups~\cite{CDL2022}, dihedral groups~\cite{HKK2022,HKK2025}, dicyclic groups~\cite{DLY2025},
and semi-dihedral groups~\cite{Yu2026}, and others.

In the theory of skew morphisms, it is an important problem to construct, classify and enumerate skew morphisms
for some nontrivial family of finite groups. In this direction, apart from the aforementioned results on skew morphisms
which are related to regular Cayley maps and cyclic complementary extensions, significant progress has been
made towards the classification of skew morphisms of cyclic groups~\cite{B2024,KN2011,KN2017,DH2019,HKK2024},
dihedral groups~\cite{HKK2022,WHYZ2019,ZD2016}, and noncyclic abelian groups~\cite{CJT2016}.

The focus of this paper is the classification of skew morphisms
on cyclic groups, a problem studied initially by Kov\'acs and Nedela in 2011~\cite{KN2011}
and which has remained open since then. In that paper, using an ingenious method from Schur rings,
the authors showed that every skew morphism $\varphi$
on $\mathbb{Z}_n$ can be decomposed into a direct product $\varphi=\varphi_1\times \varphi_2$
of two skew morphisms $\varphi_i$ on $\mathbb{Z}_{n_i}$ whenever $n=n_1n_2$ with $\gcd(n_1,n_2)=1$ and
\[
\gcd(n_1,\phi(n_2))=\gcd(n_2,\phi(n_1))=1.
\]
They continued to classify skew morphisms on the cyclic $p$-groups for odd primes $p$~\cite{KN2017}.
However, the skew morphisms of cyclic $2$-groups have not been classified yet, even though
a classification of the corresponding skew-product groups was given in~\cite{DH2019}.

Another path to advance the study of skew morphisms is to introduce new arithmetic invariants.
The first such invariant, introduced by Zhang in 2015~\cite{Zhang2015a}, is the \textit{skew-type},
which is defined to be the index $\kappa:=[A:\Ker\varphi]$, where
\[
 \Ker\varphi:=\{x\in A\mid\pi(x)\equiv1\pmod{|\varphi|}\}
 \]
 is a subgroup of $A$, termed the \textit{kernel} of $\varphi$~\cite{JS2002}.
The skew-type of a skew morphism measures how far the skew morphism is from a group automorphism.
For instance, skew morphisms of skew-type $\kappa=1$ coincide with group automorphisms,
while skew morphisms of skew-type $\kappa=2$ are generalizations of $t$-balanced skew morphisms.
The second invariant, introduced by Bachrat\'y and Jajcay in 2016~\cite{BJ2016},
is the \textit{period} of $\varphi$, which is defined to be the smallest
positive integer $\omega$ such that $\pi(\varphi^{\omega}(x))\equiv\pi(x)\pmod{|\varphi|}$ for
all $x\in A$. They showed that if $\varphi$ is a skew morphism of period $\omega$, then $\varphi^{\omega}$
is a skew morphism of period $1$ on the same group~\cite{BJ2016}. Note that skew morphisms
of period $1$ are usually called \textit{coset-preserving}~\cite{BJ2017} or \textit{smooth}~\cite{WHYZ2019}.
Motivated by these discoveries we pose the following problem.
\begin{problem}\label{prob}
For cyclic additive groups $\mathbb{Z}_n$,
\begin{enumerate}[(A)]
\item classify all skew morphisms of a prescribed skew-type $\kappa$ on $\mathbb{Z}_n$;
\item classify all skew morphisms of a prescribed period $\omega$ on $\mathbb{Z}_n$.
\end{enumerate}
\end{problem}
For cyclic additive groups $\mathbb{Z}_n$, the skew morphisms of skew-type $\kappa\leq 3$
have been documented in~\cite{CJT2007,HKZ2021}, while the skew morphisms of period $\omega=2$ remain largely
unclassified, except for a small portion constructed in~\cite{HKZ2021} under the name of
square roots of automorphisms, which are skew morphisms $\varphi$ such that $\varphi^2$
is an automorphism.

In this paper we continue the latter path and classify the non-smooth skew morphisms
on $\mathbb{Z}_n$ of skew-type $4$; see Theorem~\ref{main1}. Since such skew morphisms
always have period $\omega=2$, as a byproduct we construct several new infinite families
of skew morphisms of period $2$, thereby making partial progress on
Problem~\ref{prob}(B). In addition, the following two theorems contain complete
information on the construction, classification, and enumeration of all skew morphisms
of skew-type $4$ on cyclic $2$-groups. We now state our main results.

\medskip
\noindent\textbf{Notation.}
For a positive integer $n$, let $\nu_2(n)$ denote the $2$-adic valuation of $n$, that is, the largest
integer $\ell$ with $2^{\ell}\mid n$. We write $n=2^{e+2}$ throughout, and we set
$q:=2^e$ when $e\ge 2$.

\begin{theorem}[Smooth case]\label{main-smooth}
The cyclic additive group $\mathbb{Z}_{2^{e+2}}$ has a smooth skew morphism of skew-type $4$
if and only if $e\geq 4$, and in that case there are precisely
\[
S(e)=2^{e+2}-32
\]
of them. Each such skew morphism is uniquely determined by a triple $(r,s,t)$ of integers with
$0\le r<2^e$, $\nu_2(r)\le e-4$, and $t$ of order $4$ modulo $2^{e-\nu_2(r)}$, via the formula
\[
\varphi(x)=x+\frac{4r\Big(\big(\sum_{i=1}^s t^{i-1}\big)^x-1\Big)}{\sum_{i=1}^s t^{i-1}-1}
\pmod{2^{e+2}},
\]
where
\begin{enumerate}[\rm(1)]
\item either $s=4r+2^{e-1}+1$ and $t=2^{e-\nu_2(r)-2}+1$ or $t=3\cdot 2^{e-\nu_2(r)-2}+1$,
\item or $s=2^{e-1}+1$ and $t=2^{e-\nu_2(r)-2}-1$ or $t=3\cdot 2^{e-\nu_2(r)-2}-1$.
\end{enumerate}
\end{theorem}

\begin{theorem}[Non-smooth case]\label{main-nonsmooth}
The cyclic additive group $\mathbb{Z}_{2^{e+2}}$ has a non-smooth skew morphism of skew-type $4$
if and only if $e\geq 3$, and in that case there are precisely
\[
N(e)=2^{e+1}+\tfrac{1}{3}\bigl(2^{2e-1}-8\bigr)
\]
of them. Each such skew morphism is uniquely determined by a quintuple $(r,s,t,u,v)$
belonging to one of the following three classes, and is given by
\[
\varphi(x)=
\begin{cases}
xt\pmod{2^{e+2}}, & x\equiv0\pmod{4},\\[2pt]
xt-t+4r+3\pmod{2^{e+2}}, & x\equiv1\pmod{4},\\[2pt]
xt-2t+4\Bigl(2r+1+(r+st)\sum\limits_{i=1}^u t^{2(i-1)}\Bigr)+2\pmod{2^{e+2}}, & x\equiv2\pmod{4},\\[2pt]
xt-3t+4s+1\pmod{2^{e+2}}, & x\equiv3\pmod{4},
\end{cases}
\]
where $0\le r,s,t<2^e$ and $u,v$ are as follows.
\begin{enumerate}[\rm(1)]
\item $t=2^{e-1}+1$, $v=2^{e-1}-1$,
\[
s-r\equiv 2^{e-2}+1\pmod{2^{e-1}}\quad\text{and}\quad
u=\begin{cases}2^{e-1}-1,&\text{$r$ even},\\ 2^e-1,&\text{$r$ odd}.\end{cases}
\]
\item $t\not\equiv 2^e-1\pmod{2^e}$, $t\equiv3\pmod{4}$, $\gamma:=\nu_2(t+1)$ with $2\le\gamma\le e-2$,
$u\in\{2^{e-\gamma-1},3\cdot 2^{e-\gamma-1}\}$, $v\equiv1\pmod{4}$ with $1\le v<2^{e-\gamma+1}$, and
\[
r=\frac{t-3}{4}+\delta 2^{e-3}+j2^{e-2},\qquad
s=\frac{3t-1}{4}+\delta 2^{e-3}-j2^{e-2}\pmod{2^e},
\]
where $j=0,1,2,3$ and $\delta\in\{1,3\}$ is determined by $t$ and $u$ as follows:
if $\gamma=2$,
\[
\delta=\begin{cases}
1,&\text{$u=2^{e-3}$ and $\frac{t+1}{4}\equiv1\pmod4$, or $u=3\cdot2^{e-3}$ and $\frac{t+1}{4}\equiv3\pmod4$,}\\
3,&\text{$u=2^{e-3}$ and $\frac{t+1}{4}\equiv3\pmod4$, or $u=3\cdot2^{e-3}$ and $\frac{t+1}{4}\equiv1\pmod4$,}
\end{cases}
\]
while if $\gamma\ge3$,
\[
\delta=\begin{cases}
3,&\text{$u=2^{e-\gamma-1}$ and $\frac{t+1}{4}\equiv1\pmod{2^\gamma}$, or $u=3\cdot2^{e-\gamma-1}$ and $\frac{t+1}{4}\equiv3\pmod{2^\gamma}$,}\\
1,&\text{$u=2^{e-\gamma-1}$ and $\frac{t+1}{4}\equiv3\pmod{2^\gamma}$, or $u=3\cdot2^{e-\gamma-1}$ and $\frac{t+1}{4}\equiv1\pmod{2^\gamma}$.}
\end{cases}
\]
\item $t=2^{e-1}-1$, $v=1$, $u\in\{1,3\}$,
\[
r=2^{e-3}(\delta+1)-1+2^{e-2}j,\qquad
s=2^{e-3}(\delta-1)+2^{e-1}-2^{e-2}j,
\]
where $j=0,1,2,3$ and
\[
\delta=\begin{cases}
1,&\text{$u=1$ and $j=0,2$, or $u=3$ and $j=1,3$,}\\
3,&\text{$u=1$ and $j=1,3$, or $u=3$ and $j=0,2$.}
\end{cases}
\]
\end{enumerate}
\end{theorem}

\begin{remark}
Combining Theorems~\ref{main-smooth} and~\ref{main-nonsmooth}, we see that $\mathbb{Z}_{2^{e+2}}$
admits a skew morphism of skew-type $4$ if and only if $e\ge 3$, and the total number of such skew
morphisms is $S(e)+N(e)$.
\end{remark}

\medskip
\noindent\textbf{Organization.}
Section~2 collects the necessary preliminaries on skew morphisms and the elementary number-theoretic
lemmas we shall use. In Section~3, working with an arbitrary positive integer $n$ divisible by $4$
(not necessarily a power of $2$), we construct and classify the non-smooth skew morphisms of
skew-type $4$ on $\mathbb{Z}_n$; see Theorem~\ref{main1}. Building on this and on the earlier
classification of smooth skew morphisms due to Bachrat\'y and Jajcay (Theorem~\ref{main2}),
Sections~4 and~5 count, respectively, the smooth and non-smooth skew morphisms of skew-type $4$ on
cyclic $2$-groups; see Theorems~\ref{Scount1}, \ref{Scount2}, \ref{Ncount1} and~\ref{Ncount2}.
Combining these with Theorems~\ref{main1} and~\ref{main2} yields Theorems~\ref{main-smooth}
and~\ref{main-nonsmooth}. We close with some remarks and open problems in Section~6.

\section{Preliminaries}\label{sec:prelim}

In this section we collect the known results and prove the technical lemmas on which the rest
of the paper relies.

\subsection{Skew morphisms}

Let $\varphi$ be a skew morphism on a group $A$ with associated power function
$\pi:A\to\mathbb{Z}_m$, where $m:=|\varphi|$ is the order of $\varphi$.
The \textit{kernel} of $\varphi$ is
\[
\Ker\varphi=\{x\in A\mid \pi(x)\equiv1\pmod m\}.
\]

\begin{proposition}[\cite{CJT2007,CJT2016,JS2002}]\label{Formula}
Let $\varphi$ be a skew morphism of a finite group $A$ with power function $\pi:A\to\mathbb{Z}_m$,
where $m:=|\varphi|$. Then the following hold:
\begin{enumerate}[\rm(a)]
\item If $A$ is nontrivial, then $\Ker\varphi$ is also nontrivial.
\item If $A$ is abelian, then $\varphi(\Ker\varphi)=\Ker\varphi$; consequently, the restriction of
$\varphi$ to $\Ker\varphi$ is an automorphism of $\Ker\varphi$.
\item For all $x,y\in A$, $\pi(x)\equiv\pi(y)\pmod m$ if and only if $xy^{-1}\in\Ker\varphi$.
\item For all $x,y\in A$, $\pi(xy)\equiv\sum_{i=0}^{\pi(x)-1}\pi(\varphi^i(y))\pmod m$.
\item For all $x,y\in A$ and every integer $k$,
$\varphi^k(xy)=\varphi^k(x)\,\varphi^{\sum_{i=0}^{k-1}\pi(\varphi^i(x))}(y)$.
\end{enumerate}
\end{proposition}

Thus, $\pi$ takes exactly $\kappa:=[A:\Ker\varphi]$ distinct values in $\mathbb{Z}_m$; the index
$\kappa$ is called the \textit{skew-type} of $\varphi$. Skew morphisms of skew-type $1$ are precisely
the automorphisms, so skew morphisms generalise group automorphisms. Skew morphisms which are not
automorphisms are called \textit{proper}.

For any integer $i$, the set $\varphi^i(\Ker\varphi)$ is a subgroup of $A$, and the intersection
\[
\Core\varphi:=\bigcap_{i=1}^m\varphi^i(\Ker\varphi)
\]
is a $\varphi$-invariant normal subgroup of $A$, called the \textit{core} of $\varphi$~\cite{Zhang2015b}.
We may define the quotient skew morphism $\bar\varphi$ on $\bar A:=A/\Core\varphi$ by
\[
\overline{\varphi}(\bar x)=\overline{\varphi(x)},\qquad x\in A.
\]
Its power function $\bar\pi$ is determined by $\bar\pi(\bar x)\equiv\pi(x)\pmod{\bar m}$, where
$\bar m:=|\overline{\varphi}|$.

The power function $\pi$ is periodic: there exists a positive integer $k$ with
$\pi(\varphi^k(x))=\pi(x)$ for all $x\in A$. The smallest such integer is called the
\textit{period} of $\varphi$, denoted $\omega(\varphi)$ or simply $\omega$. Skew morphisms of period
$\omega=1$ are also called \textit{smooth} or \textit{coset-preserving}.

\begin{proposition}[\cite{WHYZ2019}, Theorem 4.5]\label{Period}
Let $\varphi$ be a skew morphism on $A$ of period $\omega$, and let $\overline{\varphi}$ be the
quotient skew morphism on $\overline{G}:=G/\Core\varphi$ induced by $\Core\varphi$. Then:
\begin{enumerate}[\rm(a)]
\item $\omega$ equals the order of $\overline{\varphi}$; in particular, $\omega$ is a positive divisor of $m$.
\item $\psi:=\varphi^{\omega}$ is a smooth skew morphism on $A$ whose power function
$\pi_{\psi}:A\to\mathbb{Z}_{m/\omega}^*$ is the homomorphism determined by
$\pi_{\psi}(x)=\tfrac{1}{\omega}\sum_{i=1}^{\omega}\pi(\varphi^i(x))$ for all $x\in A$.
\item $\varphi$ is smooth if and only if $\overline{\varphi}=\mathrm{id}_{\bar A}$.
\end{enumerate}
\end{proposition}

The power function $\pi_{\psi}$ of the smooth skew morphism $\psi=\varphi^{\omega}$ plays an
important role in the study of $\pi$. It is called the \textit{average function} of $\pi$, denoted
$\av:A\to\mathbb{Z}_{m/\omega}$:
\[
\av(x)=\frac{1}{\omega}\sum_{i=1}^{\omega}\pi(\varphi^i(x)),\qquad x\in A.
\]
Moreover, since $\bar\pi(\bar x)\equiv\pi(x)\pmod{\omega}$, one may define the \textit{mate function}
$\lambda:A\to\mathbb{Z}_{m/\omega}$ of $\pi$ (with respect to $\bar\pi$) by
\[
\lambda(x)=\frac{1}{\omega}\bigl(\pi(x)-\bar\pi(\bar x)\bigr),\qquad x\in A.
\]
These two functions have the following properties.

\begin{proposition}[\cite{HJ2026}]\label{av-lambda}
Let $\varphi$ be a skew morphism of a group $A$ with power function $\pi:A\to\mathbb{Z}_m$, let
$\overline{\varphi}$ be the quotient skew morphism of $\bar A:=A/\Core\pi$ induced by $\Core\pi$,
and let $\av$ and $\lambda$ be the average and mate functions of $\pi$, respectively. Then:
\begin{enumerate}[\rm(a)]
\item $\av(\varphi(x))\equiv\av(x)\pmod{m/\omega}$ for all $x\in A$.
\item $\av$ is a homomorphism from $A$ into the multiplicative group $\mathbb{Z}_{m/\omega}^*$.
\item If $\overline{\varphi}$ is an automorphism of $\bar A$, then
$\lambda(xy)=\av(y)\lambda(x)+\lambda(y)$ for all $x,y\in A$. In particular, if
$\av(x)\equiv1\pmod{m/\omega}$ for all $x\in A$, then $\lambda$ is a homomorphism from $A$ into
the additive group $\mathbb{Z}_{m/\omega}$.
\end{enumerate}
\end{proposition}

For cyclic groups, the following two classical results will be used repeatedly.

\begin{proposition}[\cite{KN2011}]\label{Auto}
Every skew morphism of the cyclic group of order $n$ is an automorphism if and only if either
$n=4$ or $\gcd(n,\phi(n))=1$.
\end{proposition}

\begin{proposition}[\cite{B2024,HKK2024}]\label{main2-smooth}
Every skew morphism of the cyclic group of order $n$ is smooth if and only if $n=2^e n_1$ with
$0\le e\le 4$ and $n_1$ an odd square-free positive integer.
\end{proposition}

\subsection{Elementary number theory}

\begin{proposition}[Lifting the Exponent Formula]\label{LTE}
Let $x,y\in\mathbb{Z}$ be distinct odd integers and let $n$ be a positive integer. Then:
\begin{enumerate}[\rm(a)]
\item If $2\mid(x-y)$ and $n$ is even, then
$\nu_2(x^n-y^n)=\nu_2(x-y)+\nu_2(x+y)+\nu_2(n)-1$.
\item If $2\mid(x-y)$ and $n$ is odd, then $\nu_2(x^n-y^n)=\nu_2(x-y)$.
\item If $4\mid(x-y)$, then $\nu_2(x^n-y^n)=\nu_2(x-y)+\nu_2(n)$.
\end{enumerate}
\end{proposition}

\begin{proposition}\label{Order4}
Let $a$ be a positive integer. The multiplicative group $\mathbb{Z}_{2^a}^*$ contains an element of
order $4$ if and only if $a\ge 4$, in which case its order-$4$ elements are precisely
\[
2^{a-2}-1,\quad 2^{a-2}+1,\quad 3\cdot2^{a-2}-1,\quad 3\cdot2^{a-2}+1\pmod{2^a}.
\]
\end{proposition}

\begin{proposition}\label{Value}
For any odd integer $t$ and any positive integer $h$,
\[
\nu_2\Bigl(\sum_{i=1}^h t^{2(i-1)}\Bigr)=\nu_2(h).
\]
\end{proposition}

\section{Construction and classification}\label{sec:class}

Throughout this section $n$ denotes an arbitrary positive integer divisible by $4$ (not necessarily a
power of $2$). We construct and classify the non-smooth skew morphisms of skew-type $4$ on
$\mathbb{Z}_n$.

We begin by determining the power function of such a skew morphism.

\begin{lemma}\label{PF}
Let $\varphi$ be a non-smooth skew morphism of skew-type $4$ on the cyclic additive group
$\mathbb{Z}_n$. Then the order $m$ of $\varphi$ is even and the power function
$\pi:\mathbb{Z}_n\to\mathbb{Z}_m$ of $\varphi$ is given by
\[
\pi(x)=
\begin{cases}
1, & x\equiv0\pmod4,\\
1+2u, & x\equiv1\pmod4,\\
1+2u(1+v), & x\equiv2\pmod4,\\
1+2u(1+v+v^2), & x\equiv3\pmod4,
\end{cases}
\]
where $u\in\mathbb{Z}_{m/2}$ and $v\in\mathbb{Z}_{m/2}^*$ satisfy
\begin{enumerate}[\rm(a)]
\item $u(1+v)\not\equiv0\pmod{m/2}$,
\item $u(1+v+v^2+v^3)\equiv0\pmod{m/2}$.
\end{enumerate}
\end{lemma}

\begin{proof}
Since $\varphi$ is non-smooth of skew-type $4$, the quotient skew morphism $\bar\varphi$ induced by
$\Ker\varphi=\langle 4\rangle$ is a proper skew morphism on
$\mathbb{Z}_n/\Ker\varphi\cong\mathbb{Z}_4$. By Proposition~\ref{Auto}, $\bar\varphi$ is the unique
non-identity automorphism of $\mathbb{Z}_4$, namely
$\bar\varphi=(\bar 0)(\bar 2)(\bar 1,\bar 3)$, where $\bar x=x+\Ker\varphi$.
Set $u:=\lambda(1)$ and $v:=\av(1)$. Then $\gcd(v,m/2)=1$, and by Proposition~\ref{av-lambda},
\begin{align*}
\lambda(2)&=\av(1)\lambda(1)+\lambda(1)=u(1+v),\\
\lambda(3)&=\av(2)\lambda(1)+\lambda(2)=v^2u+u(1+v)=u(1+v+v^2),\\
\lambda(4)&=\av(3)\lambda(1)+\lambda(3)=v^3u+u(1+v+v^2)=u(1+v+v^2+v^3).
\end{align*}
Since $\pi(x)=1+2\lambda(x)$, we obtain the stated values of $\pi(1),\pi(2),\pi(3),\pi(4)$;
Proposition~\ref{Formula}(c) then gives $\pi(x)$ for all $x\in\mathbb{Z}_n$. Finally,
\[
1=\pi(0)=\pi(4)=1+2u(1+v+v^2+v^3)\pmod m
\]
forces $u(1+v+v^2+v^3)\equiv0\pmod{m/2}$; and since $\pi(1),\pi(2),\pi(3),\pi(4)$ are pairwise
distinct modulo $m$, we also have $u(1+v)\not\equiv0\pmod{m/2}$.
\end{proof}

We now construct all such skew morphisms.

\begin{theorem}\label{main1}
Let $n$ be a positive integer divisible by $4$. The non-smooth skew morphisms $\varphi$ of
skew-type $4$ on $\mathbb{Z}_n$ are in bijective correspondence with the quintuples $(r,s,t,u,v)$ of integers with
\[
0\le r,s,t<n/4,\qquad 1<u,v<m/2,\qquad \gcd(n/4,t)=1,\quad\gcd(m/2,v)=1,
\]
satisfying the system
\begin{enumerate}[\rm(a)]
\item $u(1+v)\not\equiv0\pmod{m/2}$,
\item $u(1+v+v^2+v^3)\equiv0\pmod{m/2}$,
\item $t^{2u}\equiv1\pmod{n/4}$,
\item $\sigma(tv+1)\equiv0\pmod{n/4}$,
\item $\sigma t^2\equiv\sigma v^2\pmod{n/4}$,
\item $\sigma(v+1)v\equiv\sigma(t-1)\pmod{n/4}$,
\item $\sigma(1+v)\equiv\sigma(1-t)t\pmod{n/4}$,
\item $\sigma t\equiv-r-s+t-1\pmod{n/4}$,
\item $\sigma(1-t^2)\equiv2s-2r-t-1\pmod{n/4}$,
\item $\sigma(1+v)\equiv(\sigma+2r+1)\sum\limits_{i=1}^{2u}t^{i-1}\pmod{n/4}$,
\end{enumerate}
where $\sigma=(r+st)\sum_{i=1}^u t^{2(i-1)}$, and $m$ is the smallest even positive integer with
\[
m(rt+s)\sum_{i=1}^d t^{2(i-1)}\equiv0\pmod{n/4}.
\]
Moreover, if $(r,s,t,u,v)$ is such a solution, then $\varphi$ and $\pi$ are given by
\begin{equation}\label{Skew}
\begin{cases}
\varphi(4i)=4ti,\\
\varphi(4i+1)=4ti+4r+3,\\
\varphi(4i+2)=4ti+4(\sigma+2r+1)+2,\\
\varphi(4i+3)=4ti+4s+1,
\end{cases}
\quad
\begin{cases}
\pi(4i)=1,\\
\pi(4i+1)=1+2u,\\
\pi(4i+2)=1+2u(1+v),\\
\pi(4i+3)=1+2u(1+v+v^2).
\end{cases}
\end{equation}
\end{theorem}

\begin{proof}
We first derive the system (a)--(j) from the skew morphism axioms, and then verify the converse.

\medskip\noindent\textit{Necessity.}
Suppose $\varphi$ is a non-smooth skew morphism of skew-type $4$ on $\mathbb{Z}_n$. Then
$\Ker\varphi=\langle 4\rangle$, so $n\ge 8$ is divisible by $4$, and the induced skew morphism
$\overline{\varphi}$ on $\mathbb{Z}_n/\Ker\varphi$ is a non-identity skew morphism on
$\mathbb{Z}_4$. By Proposition~\ref{Auto},
$\overline{\varphi}=(\bar 0)(\bar 2)(\bar 1,\bar 3)$. By Proposition~\ref{Formula}(b) the
restriction of $\varphi$ to $\Ker\varphi$ is an automorphism, so we may write
\[
\varphi(1)=4r+3,\qquad\varphi(3)=4s+1,\qquad\varphi(4)=4t,\qquad r,s,t\in\mathbb{Z}_{n/4}.
\]
By Lemma~\ref{PF}, $m:=|\varphi|$ is even and the power function takes the stated form, with
$u\in\mathbb{Z}_{m/2}$ and $v\in\mathbb{Z}_{m/2}^*$ satisfying (a) and (b).

By induction on $k\ge 1$,
\begin{align}
\varphi^k(1)&=4(rt^{k-1}+st^{k-2}+\cdots+st+r)+3, && \text{$k$ odd,}\\
\varphi^k(1)&=4(rt^{k-1}+st^{k-2}+\cdots+rt+s)+1, && \text{$k$ even.}
\end{align}
Since $\varphi(1)+\varphi(4)=\varphi(4)+\varphi(1)=\varphi(1)+\varphi^{2u+1}(4)$, we get
$t^{2u}\equiv1\pmod{n/4}$, which is condition (c). Using this and the above formulae,
\begin{align*}
\varphi^{2u+1}(1)&=4(r+\sigma)+3,\\
\varphi^{2u(1+v)+1}(1)&=4(r+\sigma(1+v))+3,\\
\varphi^{2u(1+v+v^2)+1}(1)&=4(r+\sigma(1+v+v^2))+3,
\end{align*}
where $\sigma=st^{2u-1}+rt^{2u-2}+\cdots+st+r$. Hence
\[
\varphi(2)=\varphi(1)+\varphi^{\pi(1)}(1)=\varphi(1)+\varphi^{2u+1}(1)=4(\sigma+2r+1)+2,
\]
and therefore
\[
\varphi(4i)=4it,\quad\varphi(4i+1)=4it+4r+3,\quad
\varphi(4i+2)=4it+4(\sigma+2r+1)+2,\quad\varphi(4i+3)=4it+4s+1.
\]
A similar computation gives
\begin{align*}
\varphi^{2u+1}(2)&=4(\sigma+2r+1)(\tau+1)+2,\\
\varphi^{2u(1+v)+1}(2)&=4(\sigma+2r+1)\bigl((1+v)\tau+1\bigr)+2,\\
\varphi^{2u(1+v+v^2)+1}(2)&=4(\sigma+2r+1)\bigl((1+v+v^2)\tau+1\bigr)+2,\\
\varphi^{2u+1}(3)&=4(\sigma t+s)+1,\\
\varphi^{2u(1+v)+1}(3)&=4(\sigma t(1+v)+s)+1,\\
\varphi^{2u(1+v+v^2)+1}(3)&=4(\sigma t(1+v+v^2)+s)+1,
\end{align*}
where $\tau:=\sum_{j=1}^{2u}t^{j-1}$.

We now extract the congruences (d)--(j) by expanding
$\varphi(x+y)=\varphi(x)+\varphi^{\pi(x)}(y)$ for the sixteen ordered pairs $(x,y)$ with
$x,y\in\{1,2,3,4\}$. We group the computations by the residue class of $x$.

\smallskip
\emph{Pairs starting with $x=1$.}
From $4s+1=\varphi(1+2)=\varphi(1)+\varphi^{2u+1}(2)$ and
$4t=\varphi(1+3)=\varphi(1)+\varphi^{2u+1}(3)$ we obtain (h) together with
\begin{equation}\label{cond2}
s-r-1\equiv(\sigma+2r+1)(\tau+1)\pmod{n/4}.
\end{equation}

\smallskip
\emph{Pairs starting with $x=2$.}
Comparing $\varphi(2+y)=\varphi(2)+\varphi^{2u(1+v)+1}(y)$ for $y=1,2,3$ gives
\begin{align}
\sigma(1+v)&\equiv s-r-1-(\sigma+2r+1)\pmod{n/4},\label{cond4}\\
t-1&\equiv(\sigma+2r+1)(1+v)\tau+2(\sigma+2r+1)\pmod{n/4},\label{cond5}\\
\sigma(1+v)t&\equiv-s-r+t-1-\sigma\pmod{n/4}.\label{cond6}
\end{align}

\smallskip
\emph{Pairs starting with $x=3$.}
Comparing $\varphi(3+y)=\varphi(3)+\varphi^{2u(1+v+v^2)+1}(y)$ for $y=1,2,3$ gives
\begin{align}
t-r-s-1&\equiv\sigma(1+v+v^2)\pmod{n/4},\label{cond7}\\
t+r-s&\equiv(\sigma+2r+1)\bigl((1+v+v^2)\tau+1\bigr)\pmod{n/4},\label{cond8}\\
t-2s&\equiv\sigma t(1+v+v^2)-(\sigma+2r+1)\pmod{n/4}.\label{cond9}
\end{align}

\smallskip
\emph{Reduction to (a)--(j).}
Equating \eqref{cond2} with \eqref{cond4} gives (j). Equating (h) with \eqref{cond7} gives
\begin{equation}\label{tem1}
\sigma(1+v+v^2)\equiv\sigma t\pmod{n/4}.
\end{equation}
Note that \eqref{cond8} follows from (j), \eqref{cond6} and \eqref{tem1}, so \eqref{cond8} is
redundant. Inserting (j) into \eqref{cond5} gives
$t-1\equiv2(\sigma+2r+1)+\sigma(1+v)^2\pmod{n/4}$, which together with \eqref{cond4} gives
\begin{equation}\label{tem2}
\sigma(1+v)(1-v)\equiv2s-2r-t-1\pmod{n/4}.
\end{equation}
Inserting \eqref{tem1} into \eqref{cond9} gives (i). Subtracting \eqref{cond4} from (i) gives
$\sigma t^2+\sigma(1+v)=t-r-s-1$, which combined with (h) gives (g). Comparing \eqref{tem2}
with (i) yields
$\sigma(1+v)(1-v)\equiv\sigma(1-t^2)\equiv\sigma(1-t)(1+t)\pmod{n/4}$, which together with (g)
gives (f). Combining (f) with \eqref{tem1} gives
\[
\sigma t^2=\sigma(1+v+v^2)t=\sigma(1-t)t^2+\sigma tv^2=\sigma t^2-\sigma t^3+\sigma tv^2\pmod{n/4},
\]
which reduces to (e) since $\gcd(t,n/4)=1$. Finally, comparing (h) with \eqref{cond6} gives (d).
This completes the derivation of (a)--(j).

\medskip\noindent\textit{Sufficiency.}
Conversely, suppose $(r,s,t,u,v)$ satisfies (a)--(j). Define $\varphi$ and $\pi$ by \eqref{Skew}.
By induction on $k$,
\[
\varphi^{2k+1}(4i+1)=4(it^{2k+1}+rt^{2k}+st^{2k-1}+\cdots+st+r)+3,
\]
so
\begin{align}
\varphi^{2u+1}(4i+1)&=4(it+r+\sigma)+3,\label{power11}\\
\varphi^{2u(1+v)+1}(4i+1)&=4(it+r+\sigma(1+v))+3,\label{power12}\\
\varphi^{2u(1+v+v^2)+1}(4i+1)&=4(it+r+\sigma(1+v+v^2))+3.\label{power13}
\end{align}
Similarly,
\[
\varphi^{2k+1}(4i+3)=4(it^{2k+1}+st^{2k}+rt^{2k-1}+\cdots+rt+s)+1,
\]
hence
\begin{align}
\varphi^{2u+1}(4i+3)&=4(it+\sigma t+s)+1,\label{power21}\\
\varphi^{2u(1+v)+1}(4i+3)&=4(it+\sigma t(1+v)+s)+1,\label{power22}\\
\varphi^{2u(1+v+v^2)+1}(4i+3)&=4(it+\sigma t(1+v+v^2)+s)+1.\label{power23}
\end{align}
Finally,
\[
\varphi^{2k+1}(4i+2)=4\Bigl(it^{2k+1}+(\sigma+2r+1)\sum_{j=1}^{2k+1}t^{j-1}\Bigr)+2,
\]
so
\begin{align}
\varphi^{2u+1}(4i+2)&=4\bigl(it+(\sigma+2r+1)(\tau+1)\bigr)+2,\label{power31}\\
\varphi^{2u(1+v)+1}(4i+2)&=4\bigl(it+(\sigma+2r+1)((1+v)\tau+1)\bigr)+2,\label{power32}\\
\varphi^{2u(1+v+v^2)+1}(4i+2)&=4\bigl(it+(\sigma+2r+1)((1+v+v^2)\tau+1)\bigr)+2.\label{power33}
\end{align}
We verify $\varphi(x+y)=\varphi(x)+\varphi^{\pi(x)}(y)$ for representatives of the four residue
classes of $x$ modulo $4$; the verification splits into the following four cases.

\begin{case}[A: $x\equiv0\pmod4$]
For $y=4j,4j+1,4j+2,4j+3$ respectively,
\begin{align*}
\varphi(4i)+\varphi(4j)&=4(i+j)t=\varphi(4i+4j),\\
\varphi(4i)+\varphi(4j+1)&=4(i+j)t+4r+3=\varphi(4i+(4j+1)),\\
\varphi(4i)+\varphi(4j+2)&=4(i+j)t+4(\sigma+2r+1)+2=\varphi(4i+(4j+2)),\\
\varphi(4i)+\varphi(4j+3)&=4(i+j)t+4s+1=\varphi(4i+(4j+3)).
\end{align*}
\end{case}

\begin{case}[B: $x\equiv1\pmod4$]
Using $\pi(4i+1)\equiv2u+1$ and \eqref{power11}--\eqref{power13},
\begin{align*}
\varphi(4i+1)+\varphi^{2u+1}(4j)&=4(i+j)t+4r+3=\varphi((4i+1)+4j),\\
\varphi(4i+1)+\varphi^{2u+1}(4j+1)
&=(4it+4r+3)+(4jt+4r+4\sigma+3)\\
&=4(i+j)t+4(\sigma+2r+1)+2=\varphi((4i+1)+(4j+1)),\\
\varphi(4i+1)+\varphi^{2u+1}(4j+2)
&=(4it+4r+3)+(4jt+4(\sigma+2r+1)(\tau+1)+2)\\
&=4(i+j)t+4s+1=\varphi((4i+1)+(4j+2)),\\
\varphi(4i+1)+\varphi^{2u+1}(4j+3)
&=(4it+4r+3)+(4jt+4\sigma t+4s+1)\\
&\equiv4(i+j+1)t=\varphi((4i+1)+(4j+3)).
\end{align*}
\end{case}

\begin{case}[C: $x\equiv2\pmod4$]
Using $\pi(4i+2)\equiv2u(1+v)+1$ and \eqref{power21}--\eqref{power23},
\begin{align*}
\varphi(4i+2)+\varphi^{2u(1+v)+1}(4j)&=4(i+j)t+4(\sigma+2r+1)+2=\varphi((4i+2)+4j),\\
\varphi(4i+2)+\varphi^{2u(1+v)+1}(4j+1)
&=4(i+j)t+4s+1=\varphi((4i+2)+(4j+1)),\\
\varphi(4i+2)+\varphi^{2u(1+v)+1}(4j+2)
&=4(i+j)t+4(\sigma+2r+1)\bigl(2+(1+v)\tau\bigr)+4\\
&=\varphi((4i+2)+(4j+2)),\\
\varphi(4i+2)+\varphi^{2u(1+v)+1}(4j+3)
&=4(i+j+1)t+4r+3=\varphi((4i+2)+(4j+3)).
\end{align*}
\end{case}

\begin{case}[D: $x\equiv3\pmod4$]
Using $\pi(4i+3)\equiv2u(1+v+v^2)+1$ and \eqref{power21}--\eqref{power23},
\begin{align*}
\varphi(4i+3)+\varphi^{2u(1+v+v^2)+1}(4j)&=4(i+j)t+4s+1=\varphi((4i+3)+4j),\\
\varphi(4i+3)+\varphi^{2u(1+v+v^2)+1}(4j+1)
&=4(i+j+1)t=\varphi((4i+3)+(4j+1)),\\
\varphi(4i+3)+\varphi^{2u(1+v+v^2)+1}(4j+2)
&=4(i+j)t+4(t+r)+3=\varphi((4i+3)+(4j+2)),\\
\varphi(4i+3)+\varphi^{2u(1+v+v^2)+1}(4j+3)
&=4(i+j+1)t+4(\sigma+2r+1)+2=\varphi((4i+3)+(4j+3)).
\end{align*}
\end{case}

Thus $\varphi$ is a skew morphism on $\mathbb{Z}_n$ with power function $\pi$, and by construction
its skew-type is $4$. Distinct quintuples $(r,s,t,u,v)$ clearly give distinct skew morphisms, so the
correspondence is bijective.
\end{proof}

\section{Counting the smooth skew morphisms}\label{sec:smooth}

In this section we count the smooth skew morphisms of skew-type $4$ on the cyclic additive
$2$-groups.

\medskip\noindent\textit{Strategy.}
By Bachrat\'y and Jajcay's classification of smooth skew morphisms on cyclic groups, counting the
smooth skew morphisms of skew-type $4$ on $\mathbb{Z}_{2^{e+2}}$ reduces to solving a system of
three congruences in the integer triple $(r,s,t)$. We first extract necessary conditions on the
$2$-adic valuations of $r$ and $s-1$ (Lemma~\ref{Slemma}), then classify the solutions into two
families (Theorem~\ref{Scount1}) and finally count them (Theorem~\ref{Scount2}).

Recall the classical classification of smooth skew morphisms on cyclic groups due to
Bachrat\'y and Jajcay~\cite{BJ2017}; see~\cite{HNWY2019} for a refined version.

\begin{theorem}[\cite{BJ2017}]\label{main2}
Let $n>1$ and let $\kappa>1$ be a proper divisor of $n$. The cyclic additive group $\mathbb{Z}_n$
admits a proper smooth skew morphism of skew-type $\kappa$ if and only if the system
\begin{enumerate}[\rm(a)]
\item $t$ has multiplicative order $\kappa$ in $\mathbb{Z}_m$,
\item $s^{t-1}\equiv1\pmod{n/\kappa}$,
\item $s-1\equiv r\dfrac{\bigl(\sum_{i=1}^t s^{i-1}\bigr)^{\kappa}-1}{\sum_{i=1}^t s^{i-1}-1}\pmod{n/\kappa}$,
\end{enumerate}
has an integer solution $(r,s,t)$ with $r,s\in\mathbb{Z}_{n/\kappa}$ and $t\in\mathbb{Z}_m^*$, where
$m$ is the smallest positive integer with
\[
r\sum_{i=1}^m s^{i-1}\equiv0\pmod{n/\kappa}.
\]
The corresponding skew morphism $\varphi$ and its power function $\pi:\mathbb{Z}_n\to\mathbb{Z}_m$
are given by
\begin{equation}\label{ZSM}
\varphi(x)\equiv x+\frac{r\kappa\Bigl(\bigl(\sum_{i=1}^s t^{i-1}\bigr)^x-1\Bigr)}
{\sum_{i=1}^s t^{i-1}-1}\pmod n,
\qquad
\pi(x)\equiv t^x\pmod m.
\end{equation}
\end{theorem}

Inserting $n=2^{e+2}$ and $\kappa=4$ into Theorem~\ref{main2}, we see that smooth skew morphisms
of skew-type $4$ on $\mathbb{Z}_{2^{e+2}}$ exist only for $e\ge 2$; throughout this section we
assume $e\ge 2$ and set $q:=2^e$. By Theorem~\ref{main2} such skew morphisms correspond
bijectively to triples $(r,s,t)$ with $r\in\mathbb{Z}_{2^e}$, $s\in\mathbb{Z}_{2^e}^*$ and
$t\in\mathbb{Z}_m^*$, satisfying
\begin{enumerate}[\rm(i)]
\item $\ord_m(t)=4$,
\item $s^{t-1}\equiv1\pmod q$,
\item $s-1\equiv r(1+S_t+S_t^2+S_t^3)\pmod q$,
\end{enumerate}
where $S_t:=\sum_{i=1}^t s^{i-1}$ and $m$ is the smallest positive integer with
\begin{equation}\label{scond1}
r\sum_{i=1}^m s^{i-1}\equiv0\pmod{2^e}.
\end{equation}
The minimality of $m$ forces $m$ to be a power of $2$.

\begin{lemma}\label{Slemma1}
Let $e\ge 4$, $d\le e-4$, $a=e-d\ge 4$ and $m=2^a$. Suppose $t\in\mathbb{Z}_{2^a}^*$ has
$\ord_m(t)=4$.
\begin{enumerate}[\rm(a)]
\item If $t\equiv1\pmod4$ and $\nu_2(s-1)=d+2$, then
$1+S_t+S_t^2+S_t^3\equiv4+2^{a-1}\pmod{2^a}$.
\item If $t\equiv3\pmod4$ and $\nu_2(s-1)=e-1$, then $\nu_2(1+S_t+S_t^2+S_t^3)=a-1$.
\end{enumerate}
\end{lemma}

\begin{proof}
(a) Write $x:=s-1$, so $\nu_2(x)=d+2$. By the binomial expansion,
\[
S_t=\sum_{j=0}^{t-1}(1+x)^j=\frac{(1+x)^t-1}{x}=t+\binom{t}{2}x+\binom{t}{3}x^2+\cdots.
\]
Since $\nu_2(t-1)=a-2$, we have
$\nu_2\bigl(\binom{t}{i+1}x^i\bigr)\ge(a-3)+\nu_2(x)\ge a$ for all $i\ge1$, hence
$S_t\equiv t\pmod{2^a}$. By Proposition~\ref{Order4},
$t\equiv1+2^{a-2}$ or $1+3\cdot2^{a-2}\pmod{2^a}$, so
\[
1+S_t+S_t^2+S_t^3=(1+S_t)(1+S_t^2)\equiv4+2^{a-1}\pmod{2^a}.
\]

(b) Since $\nu_2(s-1)=e-1$,
\[
S_t=\frac{s^t-1}{s-1}\equiv t+2^{e-1}\binom{t}{2}\pmod{2^e}.
\]
As $t\equiv3\pmod4$, $\binom{t}{2}$ is odd, so $S_t\equiv t+2^{e-1}\pmod{2^e}$. Since
$\nu_2(1+t)=a-2<e-1$, we get $\nu_2(1+S_t)=a-2$, and hence $\nu_2(1+S_t^2)=1$. Consequently,
\[
\nu_2(1+S_t+S_t^2+S_t^3)=\nu_2\bigl((1+S_t)(1+S_t^2)\bigr)=(a-2)+1=a-1,
\]
as required.
\end{proof}

\begin{lemma}\label{Slemma}
Suppose $m=2^a$ and $(r,s,t)$ is a solution of {\rm(i)--(iii)}. Then:
\begin{enumerate}[\rm(a)]
\item $r\not\equiv0\pmod{2^e}$.
\item $s\equiv1\pmod4$ and $s\not\equiv1\pmod{2^e}$.
\item $a=e-\nu_2(r)$.
\item $0\le\nu_2(r)\le e-4$ and $\nu_2(s-1)\ge2$.
\item If $t\equiv1\pmod4$, then $\nu_2(s-1)=\nu_2(r)+2$.
\item If $t\equiv3\pmod4$, then $\nu_2(s-1)=e-1$.
\end{enumerate}
\end{lemma}
\begin{proof}
(a) If $r\equiv0\pmod q$, then \eqref{scond1} holds for $m=1$, contradicting the existence of an
element of order $4$ in $\mathbb{Z}_m^*$.

(b) Suppose $s\equiv3\pmod4$. Since $m$ is a power of $2$ and $t\in\mathbb{Z}_m^*$, $t$ is odd, so
\[
S_t\equiv1+s(1+s)+s^3(1+s)+\cdots+s^{t-2}(1+s)\equiv1\pmod4.
\]
Thus $1+S_t+S_t^2+S_t^3\equiv0\pmod4$, and (iii) gives the contradiction
$2\equiv s-1\equiv r(1+S_t+S_t^2+S_t^3)\equiv4r\equiv0\pmod4$.

Suppose instead $s\equiv1\pmod q$. Then $S_t=t$, and (iii) becomes
$0\equiv r(1+t)(1+t^2)\pmod{2^e}$. If $t=2^{a-2}-1$ or $3\cdot2^{a-2}-1$, then
$\nu_2(1+t)=2^{a-2}$ and $\nu_2(1+t^2)=1$, forcing $e\le\nu_2(r)+a-1$, i.e.
$a=e-\nu_2(r)\le a-1$, a contradiction. If $t=2^{a-2}+1$ or $3\cdot2^{a-2}+1$, then
$\nu_2(1+t)=\nu_2(1+t^2)=1$, so $e\le\nu_2(r)+2$, hence $a=e-\nu_2(r)\le2$, again a contradiction.

(c) Since $s\equiv1\pmod4$, Proposition~\ref{LTE}(iii) with $x=s$, $y=1$ gives
$\nu_2(S_m)=\nu_2(m)$ where $S_m=\sum_{i=1}^m s^{i-1}$. Condition \eqref{scond1} is equivalent to
$\nu_2(r)+\nu_2(S_m)\ge e$, i.e. $\nu_2(r)+\nu_2(m)\ge e$. By (a), $\nu_2(r)<e$, so minimality of
$m$ forces $m=2^{e-\nu_2(r)}$, i.e. $a=e-\nu_2(r)$.

(d) By (c), $m=2^a$; condition (i) requires an element of order $4$ in $\mathbb{Z}_m^*$, so $a\ge4$
by Proposition~\ref{Order4}. Hence $\nu_2(r)\le e-4$. By (b), $s\equiv1\pmod4$ and
$s\not\equiv1\pmod{2^e}$, so $\nu_2(s-1)\ge2$.

(e) If $t\equiv1\pmod4$, then $\nu_2(t-1)=a-2$ by Proposition~\ref{Order4}, and by
Proposition~\ref{LTE}(a),
$\nu_2(s^{t-1}-1)=\nu_2(s-1)+a-2$. Condition (ii) gives $e\le\nu_2(s-1)+a-2$, i.e.
$\nu_2(s-1)\ge\nu_2(r)+2$. On the other hand, by (d) and Lemma~\ref{Slemma1}(a),
$\nu_2(r(1+S_t+S_t^2+S_t^3))=\nu_2(r)+2<e$, so (iii) gives
$\nu_2(s-1)=\nu_2(r)+2$.

(f) If $t\equiv3\pmod4$, then $\nu_2(t-1)=1$ and Proposition~\ref{LTE}(b) gives
$\nu_2(s^{t-1}-1)=\nu_2(s-1)+1$. Condition (ii) forces $\nu_2(s-1)\ge e-1$, while (b) gives
$\nu_2(s-1)<e$; hence $\nu_2(s-1)=e-1$.
\end{proof}

\begin{theorem}\label{Scount1}
Let $e\ge4$, $0\le\delta\le e-4$, $a=e-\delta\ge4$ and $m=2^a$.
\begin{enumerate}[\rm(1)]
\item If $t\equiv2^{a-2}+1$ or $t\equiv3\cdot2^{a-2}+1\pmod{2^a}$, then $(r,s,t)$ is a solution
of {\rm(i)--(iii)} if and only if
\begin{equation}\label{SF1}
\nu_2(r)=\delta\quad\text{and}\quad s\equiv4r+2^{e-1}+1\pmod{2^e}.
\end{equation}
\item If $t\equiv2^{a-2}-1$ or $t\equiv3\cdot2^{a-2}-1\pmod{2^a}$, then $(r,s,t)$ is a solution
of {\rm(i)--(iii)} if and only if
\begin{equation}\label{SF2}
\nu_2(r)=\delta\quad\text{and}\quad s\equiv1+2^{e-1}\pmod{2^e}.
\end{equation}
\end{enumerate}
\end{theorem}

\begin{proof}
(1) If $(r,s,t)$ is a solution, then Lemma~\ref{Slemma}(e) gives
$e-\delta=a=e-\nu_2(r)$, hence $\delta=\nu_2(r)$. Write $r=2^\delta r'$ with $r'$ odd. By
Lemma~\ref{Slemma1}(a) and Lemma~\ref{Slemma}(c), condition (iii) becomes
\[
\frac{s-1}{2^\delta}\equiv r'(1+S_t+S_t^2+S_t^3)\equiv r'(4+2^{a-1})\equiv4r'+2^{a-1}\pmod{2^a},
\]
i.e. $s-1\equiv4r+2^{e-1}\pmod{2^e}$.

Conversely, if \eqref{SF1} holds, then $\ord_m(t)=4$ by Proposition~\ref{Order4}, and
$2^{a-2}\|(t-1)$. Since $(a-2)+(\delta+2)=e$,
$s^{t-1}\equiv(1+4r+2^{e-1})^{t-1}\equiv1\pmod{2^e}$, so (ii) holds. Condition (iii) is
equivalent to \eqref{SF1} as shown above.

(2) If $(r,s,t)$ is a solution, then as in (a) we get $\delta=\nu_2(r)$. Since $t\equiv3\pmod4$,
Lemma~\ref{Slemma}(f) gives $\nu_2(s-1)=e-1$, i.e. $s\equiv1+2^{e-1}\pmod{2^e}$.

Conversely, if \eqref{SF2} holds, then $\ord_m(t)=4$ and $2\mid(t-1)$, so
$s^{t-1}\equiv(1+2^{e-1})^{t-1}\equiv1\pmod{2^e}$, i.e. (ii) holds. Since $s=1+2^{e-1}$, write
$r=2^\delta r'$ and $1+S_t+S_t^2+S_t^3=2^{a-1}s'$ with $r',s'$ odd; then
\[
s-1\equiv2^{e-1}\equiv(2^\delta r')(2^{a-1}s')\equiv r(1+S_t+S_t^2+S_t^3)\pmod{2^e},
\]
i.e. (iii) holds.
\end{proof}

\begin{theorem}\label{Scount2}
Let $e\ge2$ and $n=2^{e+2}$.
\begin{enumerate}[\rm(a)]
\item If $e=2$ or $3$, then $\mathbb{Z}_{2^{e+2}}$ has no smooth skew morphisms of skew-type $4$.
\item If $e\ge4$, then $\mathbb{Z}_{2^{e+2}}$ has precisely $2^{e+2}-32$ smooth skew morphisms of
skew-type $4$.
\end{enumerate}
\end{theorem}

\begin{proof}
By Theorem~\ref{main2}, smooth skew morphisms of skew-type $4$ on $\mathbb{Z}_{2^{e+2}}$ correspond
bijectively to solutions $(r,s,t)$ of (i)--(iii). By Theorem~\ref{Scount1} such solutions exist
only for $e\ge4$, and they split into two families.

For family (a), fix $\delta$ with $0\le\delta\le e-4$. For each such $\delta$, since
$\nu_2(r)=\delta$ and $a=e-\delta$, there are $2^{e-\delta-1}$ choices for $r$ and $2$ choices for
$t$ (namely $t\equiv2^{a-2}+1$ or $3\cdot2^{a-2}+1\pmod{2^a}$); $s$ is then uniquely determined
by $s=4r+2^{e-1}+1$. This gives
\[
\sum_{\delta=0}^{e-4}2\cdot2^{e-\delta-1}=\sum_{j=4}^e2^j=2^{e+1}-2^4
\]
solutions. Family (b) yields the same count. Hence the total number of smooth skew morphisms of
skew-type $4$ on $\mathbb{Z}_{2^{e+2}}$ is $2(2^{e+1}-2^4)=2^{e+2}-32$.
\end{proof}

\section{Counting the non-smooth skew morphisms}\label{sec:nonsmooth}

In this section we count the non-smooth skew morphisms of skew-type $4$ on the cyclic additive
$2$-groups, using the classification of Theorem~\ref{main1}.

\medskip\noindent\textit{Strategy.}
Theorem~\ref{main1} reduces the classification to the system (i)--(x) below. Our task is to
enumerate the solutions. We introduce the auxiliary quantities $A,B,T_h,\sigma,k$ and then prove
two structural lemmas (Lemmas~\ref{Nlemma1} and~\ref{Nlemma2}) that force $\sigma$ into one of
two valuation regimes. Theorem~\ref{Ncount1} then describes all solutions in three explicit
classes; Theorem~\ref{Ncount2} counts them.

Throughout this section $e\ge3$ (the cases $e=0,1$ are easily seen to yield no non-smooth
skew morphisms of skew-type $4$). Put $q:=2^e$, and let $0\le r,s,t<q$ with $\gcd(t,q)=1$. Define
\[
A:=s+rt,\qquad B:=r+st,\qquad T_h:=\sum_{i=1}^h t^{2(i-1)}\quad(h\ge1).
\]
Let $m=2k\ge4$ be the smallest even positive integer with
$rt^{m-1}+st^{m-2}+\cdots+rt+s\equiv0\pmod q$; equivalently, $k$ is the smallest positive integer
with
\begin{equation}\label{mdef}
AT_k\equiv0\pmod{2^e}.
\end{equation}
Writing $a:=\nu_2(A)$ (with $a\ge e$ if $q\mid A$), Proposition~\ref{Value} gives
$\nu_2(T_k)=\nu_2(k)$, so minimality of $k$ forces
\[
k=\begin{cases}2,&a\ge e,\\ 2^{e-a},&a<e.\end{cases}
\]
Finally, let $1\le u,v<k$ with $\gcd(v,k)=1$, set
\[
\sigma:=st^{2u-1}+rt^{2u-2}+\cdots+st+r=BT_u,
\]
and put $\alpha:=\nu_2(\sigma)$, $\ell:=\nu_2(k)$.

With this notation, the system of Theorem~\ref{main1} reads:
\begin{enumerate}[\rm(i)]
\item $u(1+v)\not\equiv0\pmod k$;
\item $u(1+v+v^2+v^3)\equiv0\pmod k$;
\item $t^{2u}\equiv1\pmod q$;
\item $\sigma(tv+1)\equiv0\pmod q$;
\item $\sigma t^2\equiv\sigma v^2\pmod q$;
\item $\sigma(v+1)v\equiv\sigma(t-1)\pmod q$;
\item $\sigma(1+v)\equiv\sigma(1-t)t\pmod q$;
\item $\sigma t\equiv-r-s+t-1\pmod q$;
\item $\sigma(1-t^2)\equiv2s-2r-t-1\pmod q$;
\item $\sigma(1+v)\equiv(\sigma+2r+1)\sum\limits_{i=1}^{2u}t^{i-1}\pmod q$.
\end{enumerate}
Putting $P:=r+s$ and $Q:=2(s-r)$, conditions (viii) and (ix) become
\[
P\equiv t-1+\sigma t,\qquad Q\equiv t+1+\sigma(1-t^2)\pmod{2^e},
\]
whence
\begin{align}
4A&=2(t+1)P+(1-t)Q=(t+1)\bigl(\sigma(t^2-4t+1)+t-1\bigr)\pmod{2^e},\label{AA}\\
4B&=2(t+1)P+(t-1)Q\equiv-(t+1)\bigl(\sigma(t^2+1)-3(t-1)\bigr)\pmod{2^e}.\label{BB}
\end{align}

\begin{lemma}\label{Nlemma1}
Suppose $(r,s,t,u,v)$ is a solution of {\rm(i)--(x)}. Then:
\begin{enumerate}[\rm(a)]
\item $\nu_2(u)+\nu_2(1+v)=\ell-1$ and $\nu_2(u)\ge e-\nu_2(t^2-1)$.
\item $\sigma(t^2-1)\equiv0\pmod{2^e}$.
\item $\sigma(v^2-1)\equiv0\pmod{2^e}$.
\item $\sigma(t+v)\equiv0\pmod{2^e}$.
\item $\sigma(v-t+2)\equiv0\pmod{2^e}$.
\item If $\alpha<e$, then $t^2\equiv v^2\equiv1$, $t+v\equiv0$ and $v-t+2\equiv0\pmod{2^e}$.
\item If $\alpha\le e-2$, then $t\equiv1$ and $v\equiv-1\pmod{2^{e-\alpha-1}}$.
\item $t\not\equiv1$ and $t\not\equiv-1\pmod{2^e}$.
\end{enumerate}
\end{lemma}

\begin{proof}
Since $v$ is odd, $\nu_2(1+v^2)=1$, so (i) and (ii) are together equivalent to
$\nu_2(u(1+v))=\ell-1$, which is the first assertion of (a). By (iii),
\[
(t^2-1)T_u=t^{2u}-1\equiv0\pmod{2^e}.
\]
Since $\nu_2(T_u)=\nu_2(u)$ and $\sigma=BT_u$, we obtain $\nu_2(u)\ge e-\nu_2(t^2-1)$, completing
(a), and also $\sigma(t^2-1)\equiv0\pmod{2^e}$, which is (b). Combining (b) with (v) and (iv)
yields (c) and (d); combining with (vi) yields (e). Assertions (f) and (g) follow from (a)--(e).

For (h): suppose $t\equiv1\pmod{2^e}$. Then $A=B=r+s$, and (ix) gives
$2(s-r)\equiv2\pmod{2^e}$, i.e. $s-r\equiv1\pmod{2^{e-1}}$, so $r+s$ is odd. By (viii),
$\sigma\equiv-r-s=-A\pmod{2^e}$, so $\sigma$ is odd, $\alpha=0$, and hence
$(u+1)A\equiv Au+A=\sigma+A\equiv0\pmod{2^e}$. Since $A$ is odd, $u\equiv-1\pmod{2^e}$ is odd. But
by (g), $2^{e-1}\mid(1+v)$; then (x) gives $\sigma(1+v)\equiv(\sigma+2r+1)2u\pmod{2^e}$, so
$2^{e-1}\mid\sigma(1+v)$, whence $2^{e-2}\mid u$, contradicting $u$ odd (as $e\ge3$).

Suppose $t\equiv-1\pmod{2^e}$. Then (ix) gives $2(s-r)\equiv0\pmod{2^e}$, so
$A=s+rt=s-r$ is divisible by $2^{e-1}$, hence $k=2$, so $u=v=1$; but then (i) gives
$u(1+v)=2\not\equiv0\pmod2$, a contradiction.
\end{proof}

\begin{lemma}\label{Nlemma2}
If $(r,s,t,u,v)$ is a solution of {\rm(i)--(x)}, then either $\nu_2(\sigma)=0$ or
$\nu_2(\sigma)=e-2$.
\end{lemma}

\begin{proof}
Put $\alpha=\nu_2(\sigma)$ and $\beta=\nu_2(t-1)$. We exclude the remaining possibilities
$1\le\alpha\le e-3$ and $\alpha\ge e-1$ in turn.

\medskip
\begin{case}[$1\le\alpha\le e-3$]
By Lemma~\ref{Nlemma1}(g), $t\equiv1\pmod4$, so $\beta\ge2$; hence $\nu_2(t+1)=1$ and
$\nu_2(t^2+1)=1$, and therefore
\[
\nu_2(t^2-4t+1)=1,\qquad\nu_2\bigl(\sigma(t^2+1)\bigr)=\alpha+1.
\]
Three sub-cases arise.

\emph{Sub-case $\alpha+1>\beta$.} Then
$\nu_2\bigl(\sigma(t^2-4t+1)+t-1\bigr)=\beta$, so by \eqref{AA} $\nu_2(A)=\beta-1$ and
$\ell=e-\beta+1$. By Lemma~\ref{Nlemma1}(g), $2^{e-\alpha-1}\mid(v+1)$, so
$\nu_2(v+1)\ge e-\alpha-1\ge2$. On the other hand, (iii) gives $\nu_2(u)\ge e-\beta-1$, so by
Lemma~\ref{Nlemma1}(a),
$\nu_2(v+1)=\ell-1-\nu_2(u)\le(e-\beta+1)-1-(e-\beta-1)=1$, a contradiction.

\emph{Sub-case $\alpha+1<\beta$.} Then \eqref{AA} and \eqref{BB} give
$\nu_2(A)=\nu_2(B)=\alpha$, so $\sigma=BT_u$ forces $u$ odd. If $\beta\le e-2$, then (iii) and
$\nu_2(t^2-1)=\beta+1$ give $\nu_2(u)\ge e-\beta-1\ge1$, contradiction. If $\beta=e-1$, then
$t=2^{e-1}+1$; since $u$ is odd,
\[
\sum_{i=1}^{2u}t^{i-1}\equiv2u+2^{e-1}\pmod{2^e},
\]
whose $2$-adic valuation equals $1$. On the right-hand side of (x), $\sigma+2r+1$ is odd, so the
right side has valuation $1$; on the left, by Lemma~\ref{Nlemma1}(g),
$\nu_2(\sigma(1+v))\ge\alpha+e-\alpha-1=e-1>1$, a contradiction.

\emph{Sub-case $\alpha+1=\beta$.} As above, $\nu_2(A),\nu_2(B)\ge\alpha-1$, and since
$\sigma=BT_u$ we get $\alpha\ge\alpha-1+\nu_2(u)$, so $\nu_2(u)\le1$. On the other hand, (iii)
gives $\nu_2(u)\ge e-\beta-1=e-\alpha-2$. If $\alpha\le e-4$, this gives $\nu_2(u)\ge2$, a
contradiction; hence $\alpha=e-3$, $\beta=e-2$, $\nu_2(u)=1$. Then $k=2^{e-\alpha}=8$, and by
Lemma~\ref{Nlemma1}(a), $\nu_2(1+v)=\ell-1-\nu_2(u)=3-1-1=1$. This contradicts
Lemma~\ref{Nlemma1}(g), which gives $v\equiv-1\pmod4$.
\end{case}

\medskip
\begin{case}[$\alpha\ge e-1$]
Let $\varepsilon:=\nu_2(t^2-1)\le e$. If $\varepsilon<e$, then \eqref{AA} and \eqref{BB} give
$\nu_2(A)=\nu_2(B)=\varepsilon-2$, so $\ell=e-\varepsilon+2$ and
$\alpha=\varepsilon-2+\nu_2(u)$. The assumption $\alpha\ge e-1$ gives
$\nu_2(u)\ge e-\varepsilon+1$. Since $u<k$ and $\nu_2(k)=e-\varepsilon+2$, we get
$\nu_2(u)=e-\varepsilon+1$; then Lemma~\ref{Nlemma1}(a) gives $\nu_2(1+v)=0$, contradicting $v$
odd.

Hence $\varepsilon=e$, i.e. $t^2\equiv1\pmod q$. Then \eqref{AA} gives $2^{e-2}\mid A$, so
$\nu_2(k)\le2$. If $\nu_2(k)=1$, then $k=2$, $u=1$, and Lemma~\ref{Slemma1}(a) requires
$\nu_2(1+v)=0$, contradiction. If $\nu_2(k)=2$, then Lemma~\ref{Slemma1}(a) forces $u$ odd and
$v\equiv1\pmod4$; hence $u\in\{1,3\}$ and $v=1$. Then $\nu_2(T_u)=\nu_2(u)=0$, so $T_u$ is odd.
Since $\nu_2(\sigma)\ge e-1$, (x) reduces modulo $2^{e-1}$ to
$\sigma\equiv(\sigma+2r+1)T_u\pmod{2^{e-1}}$, whose left side is $0$ and whose right side is odd,
a contradiction.
\end{case}
\end{proof}

\begin{theorem}\label{Ncount1}
For each $e\ge3$, the system {\rm(i)--(x)} has precisely three classes of solutions
$(r,s,t,u,v)$:
\begin{enumerate}[\rm(1)]
\item $t=2^{e-1}+1$, $v=2^{e-1}-1$,
\[
s-r\equiv2^{e-2}+1\pmod{2^{e-1}},\qquad
u=\begin{cases}2^{e-1}-1,&\text{$r$ even},\\ 2^e-1,&\text{$r$ odd}.\end{cases}
\]
\item $t\not\equiv2^e-1\pmod{2^e}$, $t\equiv3\pmod4$, $\gamma:=\nu_2(t+1)\in[2,e-2]$,
$u\in\{2^{e-\gamma-1},3\cdot2^{e-\gamma-1}\}$, $v\equiv1\pmod4$ with $1\le v<2^{e-\gamma+1}$, and
\[
r=\frac{t-3}{4}+\delta 2^{e-3}+j2^{e-2},\qquad
s=\frac{3t-1}{4}+\delta 2^{e-3}-j2^{e-2}\pmod{2^e},
\]
where $j=0,1,2,3$ and $\delta\in\{1,3\}$ is determined by $t$ and $u$ as follows: for $\gamma=2$,
\[
\delta=\begin{cases}
1,&\text{$u=2^{e-3}$ and $\frac{t+1}{4}\equiv1\pmod4$, or $u=3\cdot2^{e-3}$ and $\frac{t+1}{4}\equiv3\pmod4$,}\\
3,&\text{$u=2^{e-3}$ and $\frac{t+1}{4}\equiv3\pmod4$, or $u=3\cdot2^{e-3}$ and $\frac{t+1}{4}\equiv1\pmod4$,}
\end{cases}
\]
while for $\gamma\ge3$,
\[
\delta=\begin{cases}
3,&\text{$u=2^{e-\gamma-1}$ and $\frac{t+1}{4}\equiv1\pmod{2^\gamma}$, or $u=3\cdot2^{e-\gamma-1}$ and $\frac{t+1}{4}\equiv3\pmod{2^\gamma}$,}\\
1,&\text{$u=2^{e-\gamma-1}$ and $\frac{t+1}{4}\equiv3\pmod{2^\gamma}$, or $u=3\cdot2^{e-\gamma-1}$ and $\frac{t+1}{4}\equiv1\pmod{2^\gamma}$.}
\end{cases}
\]
\item $t=2^{e-1}-1$, $v=1$, $u\in\{1,3\}$,
\[
r=2^{e-3}(\delta+1)-1+2^{e-2}j,\qquad
s=2^{e-3}(\delta-1)+2^{e-1}-2^{e-2}j,
\]
where $j=0,1,2,3$ and
\[
\delta=\begin{cases}
1,&\text{$u=1$ and $j=0,2$, or $u=3$ and $j=1,3$,}\\
3,&\text{$u=1$ and $j=1,3$, or $u=3$ and $j=0,2$.}
\end{cases}
\]
\end{enumerate}
\end{theorem}

\begin{proof}
Put $\alpha=\nu_2(\sigma)$, $\beta=\nu_2(t-1)$, $\gamma:=\nu_2(t+1)$. By Lemma~\ref{Nlemma2},
either (a) $\alpha=0$, (b) $\alpha=e-2$ with $2\le\gamma\le e-2$, or (c) $\alpha=e-2$ with
$\gamma=e-1$.

\medskip
\begin{case}[$\alpha=0$]
By Lemma~\ref{Nlemma1}(d)--(e), $t+v\equiv0$ and $v-t+2\equiv0\pmod{2^e}$, so $v\equiv-1$ and
$t\equiv1\pmod{2^{e-1}}$. By Lemma~\ref{Nlemma1}(h), $t\not\equiv1\pmod{2^e}$; hence
$t=1+2^{e-1}$. Then (iv) gives $tv+1\equiv0\pmod{2^e}$, i.e. $v=2^{e-1}-1$. Modulo $2$, (viii)
shows $A=rt+s$ is odd, so $k=2^e$. Condition (ix) gives
$2(s-r)\equiv t+1=2^{e-1}+2\pmod{2^e}$, i.e. $s-r\equiv2^{e-2}+1\pmod{2^{e-1}}$. Finally, by
(viii), $\sigma+2^{e-1}=\sigma(1+2^{e-1})=\sigma t=-(r+s)+2^{e-1}$, so $\sigma=-(r+s)$. Since
$\sigma$ is odd, $r$ even iff $s$ odd, and
\[
B=r+st=r+s+2^{e-1}s=\begin{cases}2^{e-1}-\sigma,&s\text{ odd},\\-\sigma,&s\text{ even}.\end{cases}
\]
Hence $\sigma^{-1}B=2^{e-1}-1$ if $s$ odd and $2^e-1$ if $s$ even. Since
$T_u=\sum_{i=1}^u t^{2(i-1)}=u$ and $\sigma=BT_u=Bu$,
\[
u=B^{-1}\sigma=\begin{cases}2^{e-1}-1,&r\text{ even},\\ 2^e-1,&r\text{ odd}.\end{cases}
\]
This is family (1).
\end{case}

We now assume $\alpha=e-2$. By Lemma~\ref{Nlemma1}(g), $t+v\equiv0\pmod4$ and
$v-t+2\equiv0\pmod4$, so either ($t\equiv3$, $v\equiv1$) or ($t\equiv1$, $v\equiv3$) modulo $4$.
In the second case $4\mid(1+v)$, so the left side of (x) is $0\pmod{2^e}$; since $\sigma+2r+1$ is
odd, $W:=\sum_{i=1}^{2u}t^{i-1}\equiv0\pmod{2^e}$. But (iii) gives
$\nu_2(u)\ge e-\beta-1$ with $\beta\ge2$, so by Proposition~\ref{LTE},
$\nu_2(W)=\nu_2(2u)=1+\nu_2(u)\ge e-\beta<e$, a contradiction. Hence
$t\equiv3\pmod4$ and $v\equiv1\pmod4$.

Put $\gamma=\nu_2(t+1)$. By Lemma~\ref{Nlemma1}(h), $t\not\equiv-1\pmod{2^e}$, so
$2\le\gamma\le e-1$, and
\begin{equation}\label{scond4}
\nu_2(t-1)=\nu_2(t^2+1)=\nu_2(t^2-4t+1)=\nu_2(1+v)=1.
\end{equation}
Lemma~\ref{Nlemma1}(a) then gives $\nu_2(u)=\ell-2$, so $u=k/4$ or $3k/4$.

\medskip
\begin{case}[$\alpha=e-2$, $2\le\gamma\le e-2$]
By \eqref{AA} and \eqref{scond4}, $\nu_2(A)=\gamma-1$, so $k=2^{e-\gamma+1}$. Write
$\sigma\equiv2^{e-2}\sigma'$ and $t\equiv-1+2^\gamma t'\pmod{2^e}$ with $\sigma',t'$ odd. Then
$-\sigma t\equiv2^{e-2}\sigma'\pmod{2^e}$, so (viii) reduces to
\[
r+s\equiv t-1+2^{e-2}\sigma'\pmod{2^e}.
\]
Since $\sigma(1-t^2)\equiv0\pmod{2^e}$, (ix) gives
$s-r\equiv(t+1)/2\pmod{2^{e-1}}$. Combining,
\[
r\equiv\frac{t-3}{4}+2^{e-3}\sigma'\pmod{2^{e-2}},
\]
so
\[
r=\frac{t-3}{4}+2^{e-3}\sigma'+j2^{e-2},\qquad
s=\frac{3t-1}{4}+2^{e-3}\sigma'-j2^{e-2}\pmod{2^e},
\]
where $j=0,1,2,3$.

We now show that $\sigma'$ depends only on $u$ and $t$. Write $u=2^{e-\gamma-1}u'$ with $u'$ odd.
By \eqref{BB},
\[
4B\equiv-(t+1)\bigl(\sigma(t^2+1)-3(t-1)\bigr)\pmod{2^e}.
\]
Since $\nu_2((t+1)\sigma(t^2+1))=\gamma+e-1\ge\gamma+3$,
\[
4B\equiv3(t+1)(t-1)=3\cdot2^{\gamma+1}(2^{\gamma-1}t'+1)t'\pmod{2^{\gamma+3}},
\]
so
\begin{equation}\label{ncond5}
\frac{B}{2^{\gamma-1}}\equiv3(2^{\gamma-1}t'+1)t'
=\begin{cases}t'\pmod4,&\gamma=2,\\ -t'\pmod4,&\gamma\ge3.\end{cases}
\end{equation}
Writing $Y:=t^2-1$,
\[
T_u=\sum_{i=1}^u(1+Y)^{i-1}=\frac{(1+Y)^u-1}{Y}=\binom{u}{1}+\binom{u}{2}Y+\cdots.
\]
Since $\nu_2(u)=e-\gamma-1$ and $\nu_2(Y)=\gamma+1$, we get $T_u\equiv u\pmod{2^{e-1}}$, hence
$T_u/2^{e-\gamma-1}\equiv u'\pmod4$. Multiplying this with \eqref{ncond5},
\[
\sigma'=\frac{BT_u}{2^{e-2}}\equiv\begin{cases}t'u'\pmod4,&\gamma=2,\\ -t'u'\pmod4,&\gamma\ge3.\end{cases}
\]
Thus $\sigma'$ is determined by $u$ and $t$. This is family (2).
\end{case}

\medskip
\begin{case}[$\alpha=e-2$, $\gamma=e-1$]
Here $t=2^{e-1}-1$. By \eqref{AA}, $\nu_2(A)\ge e-2$, so $k\in\{2,4\}$. If $k=2$, then
$u=v=1$, contradicting (i). Thus $k=4$, $v=1$, $u\in\{1,3\}$. Writing
$\sigma\equiv2^{e-2}\sigma'\pmod{2^e}$ with $\sigma'\in\{1,3\}$, conditions (viii) and (ix)
reduce to
\[
r+s\equiv2^{e-1}+2^{e-2}\sigma'-2\pmod{2^e},\qquad
s-r\equiv2^{e-2}\pmod{2^{e-1}},
\]
so $r\equiv2^{e-3}(1+\sigma')-1\pmod{2^{e-2}}$, i.e.
\[
r=2^{e-2}j+2^{e-3}(\sigma'+1)-1,\qquad
s=-2^{e-2}j+2^{e-3}(\sigma'-1)-1+2^{e-1},
\]
with $j=0,1,2,3$. Since $T_u\equiv u\pmod{2^e}$ and $B=r+st=2^{e-2}(2j+1)\pmod{2^e}$,
\[
2^{e-2}\sigma'=\sigma=BT_u=2^{e-2}(2j+1)u\pmod{2^e},
\]
so $\sigma'\equiv u(2j+1)\pmod4$. This is family (3).
\end{case}

Conversely, a direct verification shows that every quintuple in families (1)--(3) satisfies
(i)--(x).
\end{proof}

\begin{theorem}\label{Ncount2}
For each $e\ge3$, the number of non-smooth skew morphisms of skew-type $4$ on the cyclic additive
$2$-group $\mathbb{Z}_{2^{e+2}}$ is
\[
N(e)=2^{e+1}+\tfrac13\bigl(2^{2e-1}-8\bigr).
\]
\end{theorem}

\begin{proof}
By Theorem~\ref{main1}, non-smooth skew morphisms of skew-type $4$ on $\mathbb{Z}_{2^{e+2}}$
correspond bijectively to the solutions $(r,s,t,u,v)$ of (i)--(x), which by
Theorem~\ref{Ncount1} fall into three families.

\emph{Family (1).} $t=2^{e-1}+1$ and $v=2^{e-1}-1$ are fixed. For each choice of $r\in\mathbb{Z}_{2^e}$,
$u$ is uniquely determined, and there are exactly $2$ choices for $s$ since
$s\equiv r+2^{e-1}+1\pmod{2^{e-1}}$. This contributes $2^{e+1}$ solutions.

\emph{Family (2).} For $2\le\gamma\le e-2$, $t$ is determined by $0\le t<2^e$,
$t\equiv3\pmod4$, $\nu_2(t+1)=\gamma$, so there are $2^{e-\gamma-1}$ choices for $t$. For each
such $t$, $k=2^{e-\gamma+1}$, and there are $2$ choices for $u$, $k/4=2^{e-\gamma-1}$ choices for
$v$, and $4$ choices for $(r,s)$; hence each $t$ contributes
$2\cdot2^{e-\gamma-1}\cdot4=2^{e-\gamma+2}$ solutions. Summing over $\gamma$,
\[
\sum_{\gamma=2}^{e-2}2^{e-\gamma-1}\cdot2^{e-\gamma+1}
=2\sum_{j=2}^{e-2}2^{2j}=2^5\sum_{j=0}^{e-4}4^j=\frac{2^{2e-1}-2^5}{3}.
\]

\emph{Family (3).} $(t,v)=(2^{e-1}-1,1)$ is fixed, $u\in\{1,3\}$ ($2$ choices), and for each $u$
there are $4$ choices for $(r,s)$, contributing $8$ solutions.

Adding the three contributions,
\[
N(e)=2^{e+1}+\frac{2^{2e-1}-2^5}{3}+8
=2^{e+1}+\frac{2^{2e-1}-8}{3},
\]
as required.
\end{proof}

\section{Concluding remarks and open problems}\label{sec:concl}

Combining Theorems~\ref{main-smooth} and~\ref{main-nonsmooth}, we obtain the following picture for
the cyclic additive $2$-groups.

\begin{corollary}
The cyclic additive $2$-group $\mathbb{Z}_{2^{e+2}}$ admits a skew morphism of skew-type $4$ if and
only if $e\ge3$. In that case the total number of skew morphisms of skew-type $4$ on
$\mathbb{Z}_{2^{e+2}}$ is
\[
T(e)=S(e)+N(e)=
\begin{cases}
20,&e=3,\\[2pt]
2^{e+2}-32+2^{e+1}+\tfrac13\bigl(2^{2e-1}-8\bigr),&e\ge4.
\end{cases}
\]
\end{corollary}

The classification of Problem~\ref{prob}(A) for cyclic $2$-groups is now complete for
$\kappa\le4$: the cases $\kappa\le3$ were settled in~\cite{CJT2007,HKZ2021} and the case $\kappa=4$
is the content of the present paper. Several natural questions remain.

\begin{question}
Can the methods of this paper be extended to classify the skew morphisms of skew-type $2^a$ on
$\mathbb{Z}_{2^{e+2}}$ for arbitrary $a$, and to give a closed-form enumeration?
\end{question}

\begin{question}
The skew morphisms  treated here are of skew-type $4$, which are necessarily of period $2$. 
 Is there a uniform description of all skew morphisms of period
$2$ on cyclic groups, extending the results of~\cite{HKZ2021}?
\end{question}

\begin{question}
The method used in this paper is essentially a covering technique, as described in
\cite{WHYZ2019,Zhang2015a,Zhang2015b,HQ2026}. In the case $\kappa=4$, the induced skew morphism
$\overline{\varphi}$ on $\mathbb{Z}_n/\Ker\varphi$ is an automorphism of order $2$. 
It seems feasible to extend this
approach to classify all skew morphisms $\varphi$ on cyclic groups for which the induced skew
morphism $\overline{\varphi}$ is an automorphism.
\end{question}

We hope to address some of these questions in future work.


\end{document}